\documentclass[12pt]{amsart}
\usepackage{graphicx,amssymb,latexsym, amsmath}
\usepackage{enumerate}

\theoremstyle{plain}
\newtheorem{theorem}{Theorem}[section]
\newtheorem{corollary}[theorem]{Corollary}
\newtheorem{lemma}[theorem]{Lemma}

\newtheorem{definition}[theorem]{Definition}

\theoremstyle{remark}
\newtheorem{remark}{Remark}[section]
\newtheorem{example}[remark]{Example}

\newcommand{\ch}{\mbox{ char }}
\newcommand{\glnc}{GL(n, \mathbb{C})}
\newcommand{\ti}{\times}
\begin{document}

\title[tensor products of representations]
{On tensor products of polynomial representations}%
\author{Kevin Purbhoo}%
\address{Department of Mathematics, University
of British Columbia, Vancouver, BC  V6T 1Z2}%
\email{kevinp@math.ubc.ca}%
\author{Stephanie van Willigenburg}%
\address{Department of Mathematics, University
of British Columbia, Vancouver, BC  V6T 1Z2}%
\email{steph@math.ubc.ca}%

\thanks{The authors were supported in part by the
National Sciences and Engineering Research Council of Canada.}%

\begin{abstract} 
We determine the necessary and sufficient combinatorial
conditions for which the tensor product of two irreducible polynomial
representations of $GL(n,\mathbb{C})$ is isomorphic to another. As a consequence we discover families of Littlewood-Richardson coefficients that are non-zero, and a condition on Schur non-negativity.
\end{abstract}

\maketitle 

\section{Introduction}
\label{background} 
It is well known that the representation theory of $GL(n,\mathbb{C})$ is intimately connected to the combinatorics of partitions  \cite[Chapter 7: Appendix 2]{ECII}. Before we address the main problem in this paper that concerns the representations of $GL(n,\mathbb{C})$, we will briefly review this connection.

Recall a \emph{partition} $\lambda$ of a positive integer $m$, denoted $\lambda \vdash m$, is a list of positive integers $\lambda _1 \geq \lambda _2 \geq\cdots \geq \lambda _{\ell(\lambda)}>0$ whose sum is $m$. We call $m$ the \emph{size} of $\lambda$, the $\lambda _i$ the \emph{parts} of $\lambda$ and $\ell(\lambda)$ the \emph{length} of $\lambda$. We also let $\lambda =  0$ be the unique partition of $0$, called the \emph{empty partition} of length $0$. Every partition corresponds naturally to a \emph{(Ferrers) diagram} of \emph{shape} $\lambda$,  which consists of an array of $m$ boxes such that there are $\lambda _i$ left justified boxes in row $i$, where the rows are read from top to bottom. By abuse of notation we also denote this diagram by $\lambda$. In the following example the boxes are denoted by $\ti$.

\begin{example}
$$43211\ =\ \begin{matrix}
\ti&\ti&\ti&\ti\\
\ti&\ti&\ti\\
\ti&\ti\\
\ti\\
\ti\end{matrix}$$
\end{example}

Moreover, given partitions $\lambda , \mu$ such that  $\lambda _i\geq \mu _i$ for all $1\leq i\leq \ell(\mu)$, if we consider the boxes of $\mu$ to be situated in the top left corner of $\lambda$ then we say that $\mu$ is a \emph{subdiagram} of $\lambda$, and the
\emph{skew diagram} of shape $\lambda /\mu$
is the array of boxes contained in $\lambda$ but not in $\mu$. Again we abuse notation and denote this skew diagram by $\lambda /\mu$.

\begin{example}
$$43211/21\ =\ \begin{matrix}
\  &\  &\ti&\ti\\
\  &\ti&\ti\\
\ti&\ti\\
\ti\\
\ti\end{matrix}$$
\end{example}

Furthermore, given a (skew) diagram, we can fill the boxes with positive integers to form a \emph{tableau} $T$ and if $T$ contains 
$c_1(T)$ $1\,$s, $c_2(T)$ $2\,$s,
$\ldots$ then we say it has \emph{content} 
$c(T)=c_1(T)c_2(T)\cdots.$ 
With
this in mind we are able to state the connection between $\glnc$ and partitions of $n$ as follows.

The irreducible polynomial representations $\phi ^\lambda$ of
$GL(n,\mathbb{C})$ are indexed by  partitions $\lambda$
such that $\ell(\lambda)\leq n$ and given two irreducible polynomial representations of $GL(n,\mathbb{C})$,
$\phi ^\mu$ and $\phi ^\nu$,  one has $$\mbox{ char }(\phi ^\mu \otimes
\phi ^\nu)=\sum _{\lambda \atop \ell(\lambda) \leq n} c ^\lambda _{\mu\nu}
\mbox{ char }\phi ^\lambda$$ 
where $c^\lambda _{\mu\nu}$ is the
number of tableaux $T$ of shape $\lambda /\mu$ such that
\begin{enumerate} [(i)] 
\item the entries in the rows weakly increase from left to right 
\item the entries in the columns strictly increase from top to bottom 
\item $c(T)=\nu _1\nu _2\cdots $ 
\item when we read the entries from right to left and top to bottom the number
of $i\,$s we have read is always greater than or equal to the number
of $(i+1)\,$s we have read.  
\end{enumerate}

This  method for computing the $c^\lambda _{\mu\nu}$ is called the
\emph{Littlewood-Richardson rule}. As one might expect the $c^\lambda
_{\mu\nu}$ are called \emph{Littlewood-Richardson coefficients}.
Observe we could have equally well chosen conditions (i)-(iv) to read
\begin{enumerate}[(i)] 
\item the entries in the rows weakly increase
from \emph{right} to \emph{left} 
\item the entries in the columns
strictly increase from \emph{bottom} to \emph{top} 
\item $c(T)=\nu
_1\nu _2\cdots $ 
\item when we read the entries from \emph{left to
right and bottom to top} the number of $i\,$s we have read is always
greater than or equal to the number of $(i+1)\,$s we have read.
\end{enumerate} 
For convenience we will call this the \emph{reverse}
Littlewood-Richardson rule.

\begin{example} To illustrate both rules we now compute $c^{321}_{21,21}$. We will replace each box with the number it contains.

Using the Littlewood-Richardson rule we obtain $c^{321}_{21,21}=2$ from the tableaux 
$\begin{smallmatrix}
&&1\\
&2\\
1
\end{smallmatrix}$ and $
\begin{smallmatrix}
&&1\\
&1\\
2
\end{smallmatrix}$. Meanwhile, using the reverse Littlewood-Richardson rule we also obtain $c^{321}_{21,21}=2$ from the tableaux 
$\begin{smallmatrix}
&&1\\
&2\\
1
\end{smallmatrix}$ and 
$\begin{smallmatrix}
&&2\\
&1\\
1
\end{smallmatrix}$.
\end{example}

Another place where Littlewood-Richardson coefficients arise is in
the algebra of symmetric functions 
$\Lambda = \oplus _{m\geq 0}\Lambda ^m$, 
which is a subalgebra of $\mathbb{Z} [[x_1, x_2, \ldots ]]$ invariant under the natural action of the symmetric group.
Each $\Lambda ^m$ is spanned by $\{ s_\lambda \}_{\lambda \vdash m}$ where $s_0:=1$ and \begin{equation}
\label{schur}
s_\lambda :=\sum _T x^T.
\end{equation}
The sum is over all tableaux $T$ that satisfy conditions (i) and (ii) of the
Littlewood-Richardson rule and $x^T:=\prod _i x_i^{c_i(T)}$. For partitions $\lambda, \mu, \nu$ the  structure coefficients of these \emph{Schur functions}  satisfy $$s_\mu s_\nu = \sum _\lambda c^\lambda _{\mu\nu} s_\lambda$$
where the $c^\lambda _{\mu\nu}$ are again  Littlewood-Richardson 
coefficients.

 

%
%
%
%

Similarly we can define the algebra of symmetric polynomials on $n$
variables by setting $x_{n+1}=x_{n+2}=\cdots = 0$ above and working
with \emph{Schur polynomials} $s_\lambda (x_1,  \ldots ,x_n)$. Observe that by the definition ~\eqref{schur} if $\ell(\lambda)>n$ then $s_\lambda (x_1,  \ldots ,x_n)=0.$ The
motivation for restricting to $n$ variables is that the irreducible representations
of $\glnc$ can be indexed such that 
\begin{equation}
\label{chartoschur}
\ch \phi ^\lambda = s_\lambda (x_1,  \ldots ,x_n).
\end{equation}
See \cite{MacD, ECII} for further details.

\section{Identical tensor products} 

We now begin to address the main
problem of the paper, that is, to determine for which partitions $\lambda, \mu, \nu, \rho$ we have
\begin{equation}
\label{tensorequal}
\phi ^\lambda \otimes \phi ^\mu \cong \phi ^\nu \otimes \phi ^\rho
\end{equation}
for irreducible polynomial representations of $\glnc$.

For ease of notation, we assume $n$  is fixed throughout the
remainder of the paper. Additionally, since $s_\lambda (x_1, 
\ldots ,x_n)=0$ for $\ell(\lambda)>n$, we assume that all partitions
have at most $n$ parts.  We extend our partitions to exactly $n$ parts 
by appending a string of $n-\ell (\lambda)$ $0\,$s. For example, if
$n=4$ then $\lambda =32$ becomes $\lambda = 3200$.

We now define an operation on diagrams that will be useful later.

\begin{definition}\label{stcut} Given partitions $\lambda$ and $\mu$
and an integer $s$ such that  $0 \leq s \leq n-1$, the \emph{$s$-cut} of $\lambda$
and $\mu$ is the partition whose parts are 
\begin{gather*}
\lambda_1+\mu _1, \lambda _2 +\mu _2, \ldots , \lambda _s +\mu _s,\\
\lambda_{s+1} + \mu_n, \lambda_{s+2}+\mu_{n-1}, \ldots,
\lambda_{n-1}+\mu_{s+2}, \lambda_n + \mu_{s+1}
\end{gather*}
listed in weakly decreasing order.
\end{definition}

\begin{remark} Diagrammatically we can think of the $s$-cut of
$\lambda$ and $\mu$ as 
\begin{enumerate}[(i)] 
\item aligning the top rows of $\lambda$ and $\mu$ then 
\item cutting the diagrams $\lambda$ and $\mu$ between the $s$ and $s+1$ 
rows 
\item taking the rows of $\mu$ (or $\lambda$) below the cut and rotating them 
by $180 ^\circ$
\item appending the newly
aligned rows and sorting into weakly decreasing row length to make
a diagram.  
\end{enumerate} 
\end{remark}

\begin{example} If $n = 6$, then the $2$-cut of $432110$ and $543200$ is $973321$. This example can be viewed diagrammatically as the following.
$$\begin{matrix}
\ti&\ti&\ti&\ti\\
\ti&\ti&\ti\\\hline
\ti&\ti\\
\ti\\
\ti\\
&\end{matrix}\qquad\begin{matrix} 
\ti&\ti&\ti&\ti&\ti\\
\ti&\ti&\ti&\ti\\\hline
&\\
&\\
&\ti&\ti\\
\ti&\ti&\ti\\
\end{matrix}\quad \leadsto \quad
\begin{matrix}
\ti&\ti&\ti&\ti&\ti&\ti&\ti&\ti&\ti\\
\ti&\ti&\ti&\ti&\ti&\ti&\ti\\
\ti&\ti&\ti\\
\ti&\ti&\ti\\
\ti&\ti\\
\ti\\
\end{matrix}\ $$
\end{example}

It transpires that the $s$-cut of $\lambda$ and $\mu$ yields a
condition on Littlewood-Richardson coefficients.

\begin{lemma} 
\label{positivecoeff}
If $\lambda$, $\mu$ and $s$ are  as in Definition
\ref{stcut} and $\kappa$ is the $s$-cut of $\lambda$ and $\mu$  then 
$c^\kappa_{\lambda\mu}>0$.  
\end{lemma}

\begin{proof} 
Observe that since the Littlewood-Richardson and
reverse Littlewood Richardson rule yield the same coefficients
there must be a bijection, $\psi$, between the tableaux generated
by each. This bijection will play a key role in the proof.  

Consider
creating a tableau $T$ of shape $\kappa / \lambda$ where $\kappa
_i=\lambda_i+\mu_i$ for $1\leq i\leq s$ that will contribute towards
the coefficient $c^\kappa _{\lambda\mu}$. If we use the Littlewood-Richardson
rule then it is clear that for $1\leq i\leq s$ we must fill the boxes of the $i$-th row     with the $\mu_i$ $i\,$s. Now all that remains for us to do is to fill the remaining boxes of $T$ with $\mu_{s+1}$ $(s+1)\,$s, $\ldots\,$, $\mu_{n}$ $n\,$s. To do this we create a tableau $T'$ of shape $\kappa _{s+1}\cdots \kappa _n/\lambda _{s+1}\cdots \lambda _n = \kappa / \kappa
_1\cdots \kappa _s \lambda _{s+1}\cdots \lambda _n$ that will
contribute towards the coefficient $c^\gamma _{\alpha\beta}$ where
$\alpha = \lambda _{s+1}\cdots \lambda _n$, $\beta = \mu
_{s+1}\cdots \mu _n$ and $\gamma = \kappa _{s+1}\cdots \kappa _n$. 
We do this as follows.

Fill the box at the bottom of each column from left to right with
$\mu _{s+1}$ $1\,$s. Then repeat on the remaining boxes with the $\mu
_{s+2}$ $2\,$s. Iterate this procedure until the boxes are full. Observe
by the reverse Littlewood-Richardson rule that this filling contributes $1$ to
the coefficient $c^\gamma _{\alpha\beta}$.  Now using $\psi$, create
a tableau $T''$ of the same shape that satisfies the Littlewood-Richardson
rule and increase each entry by $s$, forming a  tableau $T'''$.
Placing the entries of $T'''$ in the naturally corresponding boxes
of $T$ we see we have a tableau that contributes $1$ to the coefficient
 $c^\kappa _{\lambda\mu}$ by the Littlewood-Richardson rule and
indeed $c^\kappa _{\lambda\mu}>0$.  
\end{proof}

\begin{definition}\label{ordset} 
If $\lambda$, $\mu$ and $s$ are  as in Definition
\ref{stcut} then the \emph{$s$-poset} of 
$\lambda$ and $\mu$ is the set of all partitions $\kappa$ such
that
\begin{enumerate}[(i)]
\item $c_{\lambda \mu}^\kappa > 0$,
\item $\kappa_i = \lambda_i + \mu_i$ for all $1 \leq i \leq s$,
\end{enumerate}
which are ordered lexicographically, that is, $\kappa > \kappa'$ if and only if there exists  some $i$, where $1\leq i\leq n$, such that
$\kappa_1 = \kappa'_1$, $\ldots\ $, $\kappa_{i-1} = \kappa'_{i-1}$ and 
$\kappa_i > \kappa'_i$.
\end{definition}

\begin{lemma}
\label{minelement}
If $\lambda$, $\mu$ and $s$ are  as in Definition
\ref{stcut} then the $s$-cut of $\lambda$ and $\mu$ is the 
unique minimal element in
the $s$-poset of $\lambda$ and $\mu$.
\end{lemma}

\begin{proof}
Let $\xi$ be any element in the $s$-poset of $\lambda$ and $\mu$ and let $U$ be any tableau that will contribute towards the coefficient $c^\xi _{\lambda\mu}$ via the Littlewood-Richardson rule. As in the proof of Lemma~\ref{positivecoeff} it is clear that for $1\leq j\leq s$ we have that $j$ appears in every box of row $j$. Now consider the natural subtableau  of shape $\xi _{s+1}\cdots \xi _n/\lambda _{s+1}\cdots \lambda _n$, which we denote by $\overline{U}$. Note that if we subtract $s$ from every entry in  $\overline{U}$ then we obtain a tableau that contributes towards $c^{(\xi _{s+1}\cdots \xi _{n})} 
_{(\lambda _{s+1}\cdots \lambda _{n})(\mu_{s+1}\cdots \mu_n)}$ via the Littlewood-Richardson rule. If we then apply the bijection $\psi$ to rearrange these new entries, we obtain a tableau $U'$ that contributes towards $c^{(\xi _{s+1}\cdots \xi _{n})} 
_{(\lambda _{s+1}\cdots \lambda _{n})(\mu_{s+1}\cdots \mu_n)}$
via the reverse Littlewood-Richardson rule.

Now let $\kappa$ be the $s$-cut of $\lambda$ and $\mu$.  Let $T$ and
$T'$ be the tableaux constructed in the proof of 
Lemma~\ref{positivecoeff}.  Recall that $T$ contributes
towards the coefficient $c^\kappa _{\lambda\mu}$ 
via the Littlewood-Richardson rule, and
that $T'$  contributes towards $c^{(\kappa _{s+1}\cdots \kappa _{n})}
_{(\lambda _{s+1}\cdots \lambda _{n})(\mu_{s+1}\cdots \mu_n)}$ via
the reverse Littlewood-Richardson rule.

We now consider transforming $T'$ into $U'$ as follows.  Since $T'$
and $U'$ both have content $\mu$, we can map the boxes of $T'$
bijectively to the boxes of $U'$ 
 such that the $k$-th box containing $i$ from the left 
in $T'$ maps to the $k$-th box containing $i$ from the left in $U'$.  This bijection
factors as follows.  First move each box in $T'$ horizontally, so that
it is in the same  column as the corresponding box in $U'$.  Then move each box  vertically to form $U'$.  By the
construction of $T'$ the entries are as left justified and low as possible,
and so this transformation necessarily moves each box rightwards and
upwards.  It follows that $\kappa$, the shape of $T'$, is lexicographically
less than or equal to $\xi$, the shape of $U'$, and we are done.
\end{proof}


Recall that $\lambda_n$ is the number of columns of length $n$ in the
diagram  $\lambda$, and thus $(\lambda_n)^n$ is a subdiagram of $\lambda$.
Define $\lambda^- := \lambda/(\lambda_n)^n$.  Notice that $\lambda^-$
is a Ferrers diagram, with at most $n-1$ rows, and the number of
columns of length $n-1$ is $\lambda^-_{n-1}$.  We therefore define
$\lambda^{--} := \lambda^-/(\lambda^-_{n-1})^{n-1}$.  Notice that
by \eqref{schur} we have the factorization
\begin{equation}
\label{lambdaminus}
s_\lambda(x_1, \dots, x_n) = (x_1 \cdots x_n)^{\lambda_n}
s_{\lambda^-}(x_1, \dots, x_n),
\end{equation}
and moreover $x_1 \cdots x_n$ does not divide 
$s_{\lambda^-}(x_1, \dots, x_n)$.

\begin{theorem} 
$\phi ^\lambda \otimes \phi ^\mu \cong \phi ^\nu \otimes \phi ^\rho$ 
as representations of $GL(n)$
if and only if
$\lambda_n+\mu_n=\nu_n+\rho_n$ and
$\{\lambda^- ,\mu^- \}=\{\nu^-, \rho^-\}$ as multisets.  
\end{theorem}

\begin{proof} We will show that
\begin{equation}
\label{schurequal}
s_\lambda(x_1,\ldots ,x_n)s_\mu(x_1,\ldots ,x_n)=  
s_\nu(x_1,\ldots ,x_n)s_\rho(x_1,\ldots ,x_n)
\end{equation}
if and only if $\lambda_n+\mu_n=\nu_n+\rho_n$ and
$\{\lambda^- ,\mu^- \}=\{\nu^-, \rho^-\}$. 
The theorem then follows, using \eqref{chartoschur}.

One direction is immediate. Suppose $\lambda_n+\mu_n=\nu_n+\rho_n$
and $\{\lambda^- ,\mu^- \}=\{\nu^-, \rho^-\}$, then
by \eqref{lambdaminus} we have 
\begin{multline*}
s_\lambda(x_1,\ldots ,x_n)s_\mu(x_1,\ldots ,x_n)\\ 
=(x_1\cdots x_n)^{\lambda_n+\mu_n} 
s_{\lambda^-}(x_1,\ldots ,x_n)s_{\mu^-}(x_1,\ldots ,x_n)\\ 
=(x_1\cdots x_n)^{\nu_n+\rho_n} 
s_{\nu^-}(x_1,\ldots ,x_n)s_{\rho^-}(x_1,\ldots ,x_n)\\ 
= s_\nu(x_1,\ldots ,x_n)s_\rho(x_1,\ldots ,x_n).
\end{multline*}

For the opposite direction, assume that \eqref{schurequal} holds.
We first show that $\lambda_n + \mu_n = \nu_n + \rho_n$.  If
they were not equal, say $\lambda_n + \mu_n > \nu_n +\rho_n$,
then by \eqref{lambdaminus}, we would have
$$(x_1\cdots x_n)^{\lambda_n + \mu_n - \nu_n -\rho_n}
s_{\lambda^-}(x_1,\ldots ,x_n)s_{\mu^-}(x_1,\ldots ,x_n)
= s_{\nu^-}(x_1,\ldots ,x_n)s_{\rho^-}(x_1,\ldots ,x_n),$$
which is impossible since $x_1\cdots x_n$ does not divide
the right hand side.  Similarly we cannot have 
$\lambda_n + \mu_n < \nu_n + \rho_n$.  Thus,  we
see furthermore that
\begin{equation}
\label{schurequalminus}
s_{\lambda^-}(x_1,\ldots ,x_n)s_{\mu^-}(x_1,\ldots ,x_n)=  
s_{\nu^-}(x_1,\ldots ,x_n)s_{\rho^-}(x_1,\ldots ,x_n).
\end{equation}

Let $S(n)$ be the assertion that the equation
\eqref{schurequalminus} holds
only if $\{\lambda^- ,\mu^- \}=\{\nu^-, \rho^-\}$.  
To complete the proof of the theorem, it remains to show that
$S(n)$ is true for all $n$.  We prove this by induction.

The base case $n=1$ is trivial, since each of $\lambda^-, \mu^-,
\nu^-, \rho^-$ is necessarily the empty partition.


Now assume that $S(1), \ldots, S(n-1)$ are true.  In particular this
assumption implies that the theorem holds for smaller values of $n$.
Furthermore, assume that \eqref{schurequalminus} holds.
Let 
\begin{xalignat*}{2}
a &:= \lambda^-_{n-1} & c &:= \nu^-_{n-1} \\
b &:= \mu^-_{n-1}  & d &:= \rho^-_{n-1}.
\end{xalignat*}
Since \eqref{schurequalminus} implies
$$
s_{\lambda^-}(x_1,\ldots ,x_{n-1})s_{\mu^-}(x_1,\ldots ,x_{n-1})=  
s_{\nu^-}(x_1,\ldots ,x_{n-1})s_{\rho^-}(x_1,\ldots ,x_{n-1}),
$$
by our inductive hypothesis we must have 
$$a +b = c+d \quad\mbox{and}\quad
\{\lambda^{--}, \mu^{--}\} =
\{\nu^{--}, \rho^{--}\}.$$
Assume without loss of generality that 
$\lambda^{--} = \nu^{--} =: \alpha$ and 
$\mu^{--} = \rho^{--} =: \beta$.  To show
that 
$\{\lambda^-, \mu^-\} = \{\nu^-, \rho^-\}$,
we need to check that $a=c$ and $b=d$, or that
  $a=d$, $b=c$ and  $\alpha = \beta$.


To show this we note that if \eqref{schurequalminus}  holds then for all $s$, $0 \leq s \leq n-1$,
the $s$-poset 
of $\lambda^-$ and $\mu^-$, must be the same as the $s$-poset
of $\nu^-$ and $\rho^-$. Thus by Lemma \ref{minelement},
the $s$-cut of $\lambda^-$ and $\mu^-$ must be the same as the
$s$-cut of $\nu^-$ and $\rho^-$.

%
%

The $s$-cut of $\lambda^-$ and $\mu^-$ has part sizes
\begin{equation*}
\begin{gathered}[t]
a+b+\alpha _j+\beta _j  \\
a+b+\alpha _{s+j}+\beta _{n-j+1} \\ 
a+\alpha _{s+1}\\ 
b+\beta _{s+1}
\end{gathered} \qquad
\begin{aligned}[t]
&1\leq j\leq s\\ 
&2\leq j\leq n-s-1 
\end{aligned}
\end{equation*}
whereas the $s$-cut of $\nu^-$ and $\rho^-$ has
part sizes 
\begin{equation*}
\begin{gathered}[t]
a+b+\alpha _j+\beta _j  \\
a+b+\alpha _{s+j}+\beta _{n-j+1} \\
c+\alpha _{s+1}\\ 
d+\beta _{s+1}.
\end{gathered} \qquad
\begin{aligned}[t]
&1\leq j\leq s\\ 
&2\leq j\leq n-s-1 
\end{aligned}
\end{equation*}
These lists must agree.  Consequently we must have
$$a+\alpha _{s+1}=c+\alpha _{s+1} \quad \text{or}\quad 
a+\alpha _{s+1}=d+\beta _{s+1},$$ 
for all $s$.
If, for any $s$, we are in the first situation, 
then $a=c$ and $b=d$ as desired. 
If not then
$$a+\alpha _{s+1}=d+\beta _{s+1} \quad \text{and} \quad 
c+\alpha _{s+1}=b+\beta _{s+1}$$
for all $0\leq s\leq n-1$.  In particular, since 
$\alpha_{n-1}=\beta_{n-1}=0$,
we have $a=d$ and $b=c$, ensuring $\alpha _j=\beta _j$ for 
$1\leq j\leq n-1$.
\end{proof}

\begin{example}
If $n=3$,
$$
\lambda = 
\begin{matrix}
\ti&\ti&\ti&\ti&\ti \\
\ti&\ti&\ti&& \\
\ti&\ti&&&
\end{matrix}\quad\text{and}\quad
\mu = 
\begin{matrix}
\ti&\ti \\
\ti&\ti \\
&
\end{matrix}
$$
then $\phi ^\lambda \otimes \phi ^\mu \cong \phi ^\nu \otimes \phi ^\rho$
if and only if
$\{\nu, \rho\}$ is equal to one of
$$\{\lambda, \mu\}\,, \quad
\left\{
\begin{smallmatrix}
\ti&\ti&\ti&\ti \\
\ti&\ti&& \\
\ti&&&
\end{smallmatrix}\ ,\ %
\begin{smallmatrix}
\ti&\ti&\ti \\
\ti&\ti&\ti \\
\ti&&
\end{smallmatrix}
\right\} \quad\text{or}\quad
\left\{
\begin{smallmatrix}
\ti&\ti&\ti \\
\ti&& \\
&&&
\end{smallmatrix}\ ,\ %
\begin{smallmatrix}
\ti&\ti&\ti&\ti \\
\ti&\ti&\ti&\ti \\
\ti&\ti&&
\end{smallmatrix}
\right\}.
$$
\end{example}

%
%
We consequently obtain a strict lower bound on $n$, in terms of the 
size of the partitions, to guarantee that \eqref{tensorequal}
has only trivial solutions.
\begin{corollary} 
Suppose $m, m'$ are non-negative integers.  
If $n >\max\{m,m'\}$, then for any partitions
$\lambda \vdash m$ and $\mu\vdash m'$, we have that
\begin{equation}\phi ^\lambda \otimes \phi ^\mu \cong \phi ^\nu \otimes \phi ^\rho\label{trivstuff}\end{equation}
has only the trivial solution $\{\nu, \rho\} = \{\lambda, \mu\}$.
If $\min\{m,m'\} \geq 2$ and $n \leq \max\{m,m'\}$, then there exist 
 $\lambda \vdash m$ and $\mu \vdash m'$ for which \eqref{trivstuff} has non-trivial solutions.
\end{corollary}
 
\section{Schur non-negativity}

A question that has received much attention recently, for example
\cite{LamPostnikovP, RhoadesSkandera},
is the question of Schur
non-negativity.  The notion of Schur non-negativity is of interest as it arises
in the study of  algebraic geometry \cite{Fomin_etal}, quantum groups \cite{LLT}, and branching problems 
in representation theory \cite{Okounkov}. 

One of the most basic Schur non-negativity questions is the following. 
Given partitions $\lambda, \mu,
\nu, \rho$ when is the difference $s_\lambda s_\mu - s_\nu s_\rho$
a non-negative linear combination of Schur functions?  Note that
if $s_\lambda s_\mu - s_\nu s_\rho$ is Schur non-negative then
the same is certainly true of the corresponding expression in finitely 
many variables
$$s_\lambda(x_1,\ldots ,x_n)s_\mu(x_1,\ldots ,x_n)- 
s_\nu(x_1,\ldots ,x_n)s_\rho(x_1,\ldots ,x_n).$$

The
following yields a test for failure of Schur non-negativity.

\begin{corollary}
For $0 \leq s \leq n-1$, let 
$\kappa = \kappa_1 \ldots \kappa_n$ 
be the $s$-cut of $\lambda$ and $\mu$, and let
$\xi = \xi_1 \ldots \xi_n$ 
be the $s$-cut of $\nu$ and $\rho$.
Form the sequences
$$
\sigma(s) := \kappa_1 \ldots \kappa_s \xi_{s+1} \ldots \xi_n
\quad\text{and}\quad
\tau(s) := \xi_1 \ldots \xi_s \kappa_{s+1} \ldots \kappa_n.
$$
If there exists an $s$ for which $\tau(s)$ is lexicographically 
greater than $\sigma(s)$, then 
$$s_\lambda(x_1,\ldots ,x_n)s_\mu(x_1,\ldots ,x_n)- 
s_\nu(x_1,\ldots ,x_n)s_\rho(x_1,\ldots ,x_n)$$
is not Schur non-negative. 
\end{corollary}

\begin{proof}
Suppose the $s$-cut of $\lambda$ and $\mu$ is not equal to the
$s$-cut of $\nu$ and $\rho$, and let $k$ be the first index in
which they differ.  If $k \leq s$, and $\xi_k > \kappa_k$, then by
the Littlewood-Richardson rule, $c_{\lambda\mu}^\xi = 0$.
On the other hand if
$k > s$ and $\kappa_k > \xi_k$, then the same is true
by Lemma~\ref{minelement}.  
In either case, by Lemma~\ref{positivecoeff}, $c_{\nu\rho}^\xi > 0$,
and thus 
$$s_\lambda(x_1,\ldots ,x_n)s_\mu(x_1,\ldots ,x_n)- 
s_\nu(x_1,\ldots ,x_n)s_\rho(x_1,\ldots ,x_n)$$
is not Schur non-negative.
\end{proof}


\begin{example}
Suppose $n=3$, and
\begin{xalignat*}{2}
\lambda &= 310  & \nu &= 220 \\
\mu &= 110  & \rho &= 200.
\end{xalignat*}
Then $\sigma(0) =  222 < 321 = \tau(0)$.  Thus we can conclude that 
$s_\lambda s_\mu - s_\nu s_\rho$ is not Schur non-negative.
On the other hand $\sigma(1) = 420 > 411 = \tau(1)$.  Thus
$s_\nu s_\rho - s_\lambda s_\mu$ is also not Schur non-negative.
\end{example}
\section{Acknowledgements} The authors are grateful to Matthew Morin
for helpful discussions, Christopher Ryan for providing some data
and Mark Skandera for suggesting the problem.  They would also like to 
thank the referee for illuminating comments.


\end{document}